\documentclass[12pt]{article}
\usepackage{amsthm}
\usepackage{amsmath}
\usepackage{enumerate}
\usepackage{amssymb}
\usepackage{latexsym}
\usepackage{amsfonts}
\usepackage{color}
\usepackage{mathrsfs}
\usepackage{epsfig}
\newtheorem{theorem}{\bf Theorem}[section]
\newtheorem{proposition}[theorem]{\bf Proposition}
\newtheorem{definition}[theorem]{\bf Definition}
\newtheorem{corollary}[theorem]{\bf Corollary}

\newtheorem{lemma}[theorem]{\bf Lemma}

\newsavebox{\savepar}

\begin{document}
	\title{Singular Fractional Choquard Equation with a Critical Nonlinearity and a Radon measure}
	\author{\small Akasmika Panda\footnote{akasmika444@gmail.com}$^{~,1}$, Debajyoti Choudhuri\footnote{dc.iit12@gmail.com}$^{~,1}$ \&  Kamel Saoudi\footnote{kmsaoudi@iau.edu.sa (Corresponding author)},$^{\,,2}$\\
	\small{$^1$\it Department of Mathematics, National Institute of Technology Rourkela, India}\\
	\small{$^2$\it Basic and Applied Scientifc Research Center, Imam Abdulrahman Bin Faisal University,,}\\ \small{\it P.O. Box 1982, 31441, Dammam, Saudi Arabia}}
		\date{}
		\maketitle
		\begin{abstract}
\noindent This article concerns about the existence of a positive SOLA (Solutions Obtained as Limits of Approximations) for the following singular critical Choquard problem  involving fractional power of Laplacian and a critical Hardy potential.
			\begin{equation}
			\begin{split}
			(-\Delta)^su-\alpha \frac{u}{|x|^{2s}}&=\lambda u+ u^{-\gamma}+\beta \left(\int_{\Omega}\frac{u^{2_b^*}(y)}{|x-y|^b}dy\right)u^{2_b^*-1}+\mu ~\text{in}~\Omega,\\
			u&>0~\text{in}~\Omega,\\
			u&= 0~\text{in}~\mathbb{R}^N\setminus\Omega.
			\end{split}
			\end{equation}
Here, $\Omega$ is a bounded domain of $\mathbb{R}^N$, $s\in (0,1)$, $\alpha,\lambda$ and $\beta$ are positive real parameters, $N>2s$, $\gamma\in (0,1)$, $0<b<\min\{N,4s\}$, $2_b^*=\frac{2N-b}{N-2s}$ is the critical exponent in the sense of Hardy–Littlewood–Sobolev inequality and $\mu$ is a bounded Radon measure in $\Omega$.\\
			\textbf{Keywords:} Choquard equation, Fractional Sobolev spaces, Radon measure, Marcinkiewicz space, Hardy potential.\\
			\textbf{AMS Classification:}  35J60, 35R11, 35A15.
		\end{abstract}
		\section{Introduction}\label{introduction}
		The nonlocal problem involving fractional Laplacian has a remarkable contribution in various fields of science. The fractional Laplacian arise in chemical reactions in liquids, diffusion in plasma, geophysical fluid dynamics, electromagnetism and is the infinitesimal generator of L\'{e}vy stable diffusion process, see \cite{Applebaum} for instance. Therefore, a considerable amount of research is carried out by a numerous scientists, engineers, mathematicians with equal interest. Elliptic PDEs involving singular nonlinearity have been studied by many authors, refer \cite{Boccardo,Canino,Haitao,Lazer,Sun} and the references therein. In all these referred works the authors have proved the existence of solution to the singular problem with approximation arguments and the solution space depends on the power of the singular term. Recently, Sun \& Zhang in \cite{Sun} explained the role of power 3 for elliptic equations with negative exponents and claimed that for exponent greater than 3, the problem does not possess a solution. Similarly, problems involving a Radon measure as a non homogeneous term are also treated with approximations, since we can always approximate a Radon measure by sequence of smooth functions. Thus, in this case one can expect the solution space with lesser degree of differentiability or/and integrability. For example, Boccardo et al. (\cite{Boccardo 2},\cite{Boccardo 1}) proved the existence of solution in $W^{1,m}_0(\Omega)$ for every $m<\frac{N(p-1)}{N-1}$ for a problem involving $p$-Laplacian and a Radon measure. Later, Kuusi et al. \cite{Kuusi} extended the work of Boccardo to fractional $p$-Laplacian set up and guaranteed a solution in $W^{\bar{s},m}(\Omega)$ for every $\bar{s}<s<1$, $m<\min\{\frac{N(p-1)}{N-s},p\}$. Further search of the literature led us to find similar problems but consisting of both a singularity and a Radon measure. The local case (with Laplace operator) of such problems has been dealt by Panda et al. in \cite{Panda} and the corresponding problem admits a weak solution in $W_0^{1,m}(\Omega)$ if $\gamma\in (0,1]$  and in $W_{loc}^{1,m}(\Omega)$ if $\gamma>1$ for all $m<\frac{N}{N-1}$. The nonlocal case (with fractional Laplacian) with a singularity and a Radon measure has been studied by Ghosh et al. in \cite{Ghosh}.\\
	In this paper we will consider the following singular fractional elliptic problem with a Choquard type critical nonlinearity and a Radon measure. The motivation to consider this work has been mentioned towards the end of this section.
		\begin{equation}\label{4p1}\tag{$P_\beta$}
		\begin{split}
		(-\Delta)^su-\alpha \frac{u}{|x|^{2s}}&=\lambda u+ u^{-\gamma}+\beta \left(\int_{\Omega}\frac{u^{2_b^*}(y)}{|x-y|^b}dy\right)u^{2_b^*-1}+\mu ~\text{in}~\Omega,\\
		u&>0~\text{in}~\Omega,\\
		u&= 0~\text{in}~\mathbb{R}^N\setminus\Omega,
		\end{split}
		\end{equation}
			where $\Omega$ is a bounded domain in $\mathbb{R}^N$ with $C^2$ boundary, $s\in(0,1)$, $N>2s$, $0<\gamma<1$, $\alpha,\beta,\lambda>0$, $b<\min\{N,4s\}$, $\mu$ is a bounded Radon measure and $(-\Delta)^s$ is the fractional Laplacian defined by 
			$$(-\Delta)^su=\text{P. V.}\int_{\mathbb{R}^N}\frac{u(x)-u(y)}{|x-y|^{N+2s}}dy.$$
Nonlinear problems involving a Choquard term draw its motivation from the Hardy-Littlewood-Sobolev inequality. Buffoni in \cite{Buffoni}, considered the following Choquard problem and shown the existence of a ground state solution. 
\begin{equation}\label{choquard term}
(-\Delta)u+V(x)u=\left(\frac{1}{|x|^b}*|u|^p\right)|u|^{p-2}u~~\text{in}~\mathbb{R}^N
\end{equation}
for $p>1$ and $N\geq3$.	S. Pekar in \cite{Pekar} studied the problem $\eqref{choquard term}$ for $p=2$ and $b=1$ as a physical model and described the quantum theory of a polaron at rest. Later, P. Choquard \cite{Choquard} used the Choquard problem of type $\eqref{choquard term}$ for the modeling of one component plasma. The nonlocal Choquard problem, i.e. the Choquard problem with fractional Laplacian, is known as nonlinear fractional Schrödinger equation with Hartree-type nonlinearity. These problems have a wide application in the quantum mechanical theory, mean field limit of weakly interacting molecules, physics of multi particle systems, etc. One can refer \cite{Chen, Liu} and the references therein for further study of fractional Choquard problem.\\
The Brezis–Nirenberg type critical Choquard problem in a bounded domain $\Omega$, that is
\begin{equation}
\begin{split}
-\Delta u&=\lambda u+ \left(\int_{\Omega}\frac{|u|^{2_b^*}(y)}{|x-y|^b}dy\right)|u|^{2_b^*-2}u ~\text{in}~\Omega,\\
u&= 0~\text{in}~\mathbb{R}^N\setminus\Omega,
\end{split}
\end{equation}
has been studied by \cite{Gao,Tang,Giacomoni Hardy}, etc. Gao \& Yang in \cite{Gao} proved the existence, nonexistence and multiplicity results for a range of $\lambda$. \\
Recently, Giacomoni et al. in \cite{Giacomoni arxiv} dealt with fractional critical Choquard problem with singular nonlinearity, i.e. they considered problem ($P_\beta$) with $\alpha,\lambda=0$ and without the Radon measure $\mu$. In \cite{Giacomoni arxiv}, the authors have explained a very weak comparison principle, established the existence of two positive weak solution and discussed about the Sobolev regularity of the solutions.\\
The problem $(P_\beta)$ involves two critical terms, the Hardy potential in the left hand side and the Choquard nonliner term in the right hand side. The nonlocal problems with a Hardy critical potential have been recently treated in \cite{Barrios,Fall,Fiscella}, etc. In 2016, Fiscella \& Pucci in \cite{Fiscella} studied the following problem with a Hardy term and proved the existence of multiple solutions with the explanation of the asymptotic behavior of solutions.
	\begin{equation}
\begin{split}
(-\Delta)^su-\alpha \frac{u}{|x|^{2s}}&=\lambda u+\theta f(x,u)+g(x,u) ~\text{in}~\Omega,\\
u&= 0~\text{in}~\mathbb{R}^N\setminus\Omega,
\end{split}
\end{equation}
where the function $f$ appears with a sub critical growth while $g$ could be either a critical term or a perturbation.\\
Motivated by the above works, in this paper, we discuss the problem $(P_\beta)$ in a bounded domain. To the best of our knowledge, this work is novel, even for the local case (i.e. for $s=1$), in the sense that in the literature there is no contribution whatsoever which indicates a study on problem involving a singular nonlinearity, Hardy potential, Choquard nonlinearity and a measure data together. We find a very less number of articles dealing with singular fractional problem with critical exponent and measure data. Amongst them, Panda et al. in \cite{Panda 1} have considered the following problem and obtained a positive SOLA via a sequence of approximating problems.
\begin{equation}
\begin{split}
(-\Delta)^su&= \frac{1}{u^\gamma}+\lambda u^{2_s^*-1}+\mu ~\text{in}~\Omega,\\
u&>0~\text{in}~\Omega,\\
u&= 0~\text{in}~\mathbb{R}^N\setminus\Omega.
\end{split}
\end{equation}
We have extended the work in \cite{Panda 1} by considering a Hardy potential and a Choquard type critical nonlinearity in $(P_\beta)$. We show the existence of a SOLA to our problem with the method of approximations. We follow the approach closely related to the approaches used in $\cite{Fiscella}$, \cite{Giacomoni arxiv} and \cite{Panda 1}. \\
\noindent Turning to the paper organization: In Section $\ref{2}$, we provide some functional settings, introduce a suitable notion of solution (SOLA) to ($P_\beta$) and further state some auxiliary results and main results. In Section $\ref{sec 2}$, we show that the approximating problem to ($P_\beta$) admits a positive weak solution for a certain range of $\beta$. Finally, in Section $\ref{sec 3}$, we prove our main result, i.e. Theorem $\ref{main theorem}$, and guarantee the existence of a SOLA to $(P_\beta)$.
\section{Functional settings and auxiliary results}\label{2}
The fractional Sobolev space $W^{s,p}(\mathbb{R}^N)$, for $1\leq p<\infty$ and for $s\in(0,1)$, is defined as
$$W^{s,p}(\mathbb{R}^N)=\left\{u\in L^p(\mathbb{R}^N):\int_{\mathbb{R}^{2N}}\frac{{|u(x)-u(y)|^p}}{|x-y|^{N+sp}}dxdy<\infty\right\}$$ and 
$$W_0^{s,p}(\Omega)=\{u\in W^{s,p}(\mathbb{R}^N): \int_{\mathbb{R}^{2N}}\frac{|u(x)-u(y)|^{p}}{|x-y|^{N+sp}}dydx<\infty,~u=0 \text{ in }\mathbb{R}^N\setminus\Omega\}$$ is a reflexive subspace of $W^{s,p}(\mathbb{R}^N)$ endowed with the following norm $$\|u\|_{W_0^{s,p}(\Omega)}^p=\int_{\mathbb{R}^{2N}}\frac{|u(x)-u(y)|^{p}}{|x-y|^{N+sp}}dydx.$$
\noindent Further, for $p=2$, we denote the space $W^{s,p}(\mathbb{R}^N)$ as $H^{s}(\mathbb{R}^N)$ and $W_0^{s,p}(\Omega)$ as $H_0^{s}(\Omega)$. Actually, $H_0^{s}(\Omega)$ is the completion of $C_0^\infty(\Omega)$ with respect to the following norm  $$\|u\|_{H_0^{s}(\Omega)}^2=\int_{\mathbb{R}^{2N}}\frac{|u(x)-u(y)|^{2}}{|x-y|^{N+2s}}dydx.$$
Moreover, the space $(H_0^{s}(\Omega),\|\cdot\|_{H_0^{s}(\Omega)})$ is a reflexive separable Hilbert space. According to Proposition 3.6 of \cite{Valdinocci}, the norms $\|\cdot\|_{H_0^{s}(\Omega)}$ and $\|(-\Delta)^{s/2}\cdot\|_{L^2(\mathbb{R}^N)}$ are norm equivalent. We now state the well known fractional Sobolev embedding theorem, Theorem 6.5 of \cite{Valdinocci}, which will be used frequently throughout this article.
\begin{theorem}\label{constant}
	Let $s\in(0,1)$ and $1\leq p<\infty$ with $N>sp$. Then there exists $C=C(s,N,p)>0$ such that for every $u\in W_0^{s,p}(\Omega)$, 
	$$\|u\|_{L^q(\Omega)}\leq C \|u\|_{W_0^{s,p}(\Omega)} $$
	for all $1\leq q\leq p_s^*=\frac{Np}{N-sp}$. Further, the embedding from $W_0^{s,p}(\Omega)$ to $L^q(\Omega)$ is compact for every $q\in[1,p_s^*)$.
\end{theorem}
\noindent Define $\mathbb{S}_{s,p}$, the best Sobolev constant in the Sobolev embedding theorem, by 
\begin{equation}\label{best}
\mathbb{S}_{s,p}=\underset{u\in W_0^{s,p}(\Omega)\setminus\{0\}}{\inf}\frac{\|u\|_{W_0^{s,p}(\Omega)}^p}{\|u\|_{L^{p_s^*}(\Omega)}^p}.
\end{equation}
 Choose $b<\min\{N,4s\}$. Let us denote $2^*_b=\frac{2N-b}{N-2s}$ and for any $u\in L^{2_s^*}(\mathbb{R}^N)$, define  $$\|u\|_C=\left(\int_{\mathbb{R}^{N}}\int_{\mathbb{R}^{N}}\frac{u^{2_b^*}(x)u^{2_b^*}(y)}{|x-y|^b}dx dy\right)^{1/22_b^*}.$$
According to Lemma 2.2 of \cite{Wang}, $\|\cdot\|_C$ is a norm equivalent to the standard norm $\|\cdot\|_{L^{2_s^*}(\mathbb{R}^N)}$ on $L^{2_s^*}(\mathbb{R}^N)$. Thus, in this sense we can say that the problem $(P_\beta)$ is a critical Choquard type problem. To understand this sense of criticalness, we need to introduce the Hardy-Littlewood-Sobolev Inequality which is the foundation of the Choquard problem of type $(P_\beta)$.
\begin{proposition}[Proposition 2.1 of \cite{Giacomoni Hardy}]\label{HLS}
	Let $t,r>1$ and $0<b<N$ with $1/t+1/r+b/N=2$. Further, assume $f\in L^t(\mathbb{R}^N)$ and $g\in L^r(\mathbb{R}^N)$. Then there exists a sharp constant $C(t,r,b,N)>0$ such that
	$$\int_{\mathbb{R}^{N}}\int_{\mathbb{R}^{N}}\frac{f(x)g(y)}{|x-y|^b}dx dy\leq C(t,r,b,N)\|f\|_{L^t(\mathbb{R}^N)}\|g\|_{L^r(\mathbb{R}^N)}.$$
\end{proposition}
\noindent For the choice $f=g=|u|^{2_b^*}$, by using the above inequality we get
\begin{equation}\label{sharp constant}
\|u\|_C^{22^*_b}\leq C(N,b)\|u\|_{L^{2_s^*}(\mathbb{R}^N)}^{22_b^*}.
\end{equation}
\noindent Define 
\begin{equation}\label{best sobolev choquard}
\mathbb{S}_{C,b}=\underset{u\in H_0^{s}(\Omega)\setminus\{0\}}{\inf}\frac{\|u\|_{H_0^{s}(\Omega)}^2}{\|u\|_{C}^2}.
\end{equation}
\begin{lemma}[Lemma 2.5 of \cite{Giacomoni Hardy}]\label{best constant achieve}
	\noindent The constant $\mathbb{S}_{C,b}$ is achieved if and only if 
	$$u=C\left(\frac{k}{k^2+|x-x_0|^2}\right)^{\frac{N-2s}{2}}$$
	where $C>0$ is a fixed constant, $x_0\in \mathbb{R}^N$ and $k\in(0,\infty)$ are parameters. Moreover, 
	$$\mathbb{S}_{C,b}=\frac{\mathbb{S}_{s,2}}{C(N,b)^{\frac{N-2s}{2N-b}}}.$$
\end{lemma}
\noindent Let us define the functional of the elliptic part of our problem $(P_\beta)$ as 
\begin{equation}\label{elliptic}
\mathbb{H}_{\alpha,\lambda}(u)=\frac{1}{2}\left(\|u\|^2_{H_0^s(\Omega)}-\alpha\|u\|_{NH}^2-\lambda\|u\|^2_{L^2(\Omega)}\right), ~\text{where}~\|u\|_{NH}^2=\int_{\mathbb{R}^N}\frac{|u(x)|^2}{|x|^{2s}}dx.
\end{equation}
Since the embedding $H_0^s(\Omega)\hookrightarrow L^2(\Omega,|x|^{-2s})$ is continuous but not compact, the Hardy term in the problem is also a critical part. To get rid of this critical term, we look for a range of $\alpha$ and $\lambda$ such that the functional $\mathbb{H}_{\alpha,\lambda}$ will become weakly lower semi continuous and coercive in $H_0^s(\Omega)$.\\
Let $\lambda_1$ be the first eigenvalue of the fractional Laplacian $(-\Delta)^s$ and hence $\lambda_1>0$. Thus, for every $\lambda<\lambda_1$ and every $u\in H_0^s(\Omega)$, we get the following inequality (refer $\cite{Fiscella}$).
\begin{equation}\label{eigen}
m_\lambda \|u\|_{H_0^s(\Omega)}^2\leq \int_{\Omega}|(-\Delta)^{s/2}u(x)|^2 dx-\lambda\int_\Omega|u(x)|^2 dx\leq M_\lambda \|u\|_{H_0^s(\Omega)}^2,
\end{equation}
where $m_\lambda=1-\frac{\lambda^-}{\lambda_1}$ and $M_\lambda=1+\frac{\lambda^+}{\lambda_1}$. The best fractional Hardy constant $C_H=C_H(N,s)>0$, defined below, plays an important role in $H_0^s(\Omega)$ given by,
\begin{equation}
C_H=\underset{u\in H_0^s(\Omega),u\neq0}{\inf}\frac{\|u\|_{H_0^s(\Omega)}^2}{\|u\|_{NH}^2}.
\end{equation}
\begin{corollary}[Corollary 2.3 of \cite{Fiscella}]\label{coercive}
	For any $\lambda\in (-\infty,\lambda_1)$ and $\alpha\in (-\infty, m_\lambda C_H)$, the functional $\mathbb{H}_{\alpha,\lambda}:H_0^s(\Omega)\rightarrow\mathbb{R}$, defined in $\eqref{elliptic}$,
	is weakly lower semi continuous and coercive in $H_0^s(\Omega)$. Furthermore, 
	$$\mathbb{H}_{\alpha,\lambda}(u)\geq \frac{1}{2}\left(m_\lambda-\frac{\alpha^+}{C_H}\right)\|u\|_{H_0^s(\Omega)}^2.$$	
\end{corollary}
\noindent The following theorem is a commonly used variational principle, known as `Ekeland Variational Principle', to prove the existence of solution to variational problems.
\begin{theorem}\label{eke}
	(Ekeland Variational Principle $\cite{Ekeland}$) Assume $H$ to be a Banach space and the function $J:H\rightarrow\mathbb{R}\cup\{+\infty\}$ is  G$\hat{a}$teaux-differentiable, lower semi continuous and bounded from below. Then for every $\epsilon>0$ and for every $u\in H$ satisfying $ J(u)\leq\inf J+\epsilon$, every $\delta>0$, there exists $v\in H$ such that $J(v)\leq J(u)$, $\|u-v\|\leq\delta$ and $\|J^\prime(v)\|^*\leq\frac{\epsilon}{\delta}$. The norms $\|.\|$ and $\|.\|^*$ are the norm of $V$ and the dual norm of $V^*$, respectively.
\end{theorem}	 
\noindent Since our problem $(P_\beta)$ involves a measure data, we do not expect the solution space to be $H_0^s(\Omega)$ but expect the solution to lie in a space with a lower degree of integrability or/and differentiability. Thus, we look for a SOLA (Solutions Obtained as Limits of Approximations). We now define the notion of solution to problem ($P_\beta$).
\begin{definition}\label{SOLA}
Let $\mathcal{M}(\Omega)$ be the set of all finite Radon measures on $\Omega$ and $\mu\in \mathcal{M}(\Omega)$. Then a function $u\in W_0^{\bar{s},m}(\Omega)$, for $\bar{s}<s$ and $m<\frac{N}{N-s}$, is said to be a SOLA to $(P_\beta)$  if
	\begin{equation}\label{SOLA weak}
	\int_{\mathbb{R}^{N}}(-\Delta)^{s/2}u\cdot(-\Delta)^{s/2}\phi-\alpha\int_\Omega\frac{u\phi}{|x|^{2s}}=\lambda\int_{\Omega}u\phi+\int_{\Omega}\frac{\phi}{u^\gamma}+\beta\int_{\Omega}\int_{\Omega} \frac{u^{2_b^*}(y)u^{2_b^*-1}(x)\phi(x)}{|x-y|^b}dxdy+\int_{\Omega} \phi d\mu, 
	\end{equation}
	for every $\phi\in C_c^\infty(\Omega)$. Further, for any $\omega\subset \subset\Omega$, there exists a $C_\omega$ such that 
	\begin{equation}\label{comapct}
	u\geq C_\omega>0.
	\end{equation}
\end{definition}
\noindent Consider a sequence $(\mu_n)\subset L^\infty(\Omega)$ which is $L^1$ bounded and converges to $\mu$ in the sense of measure as defined in the following definition.
\begin{definition}\label{measure}
Assume	$(\mu_n)\subset\mathcal{M}(\Omega)$ to be a sequence of measurable functions. Then $(\mu_n)$ converges to $\mu\in \mathcal{M}(\Omega)$ in the sense of measure if $$\int_{\Omega}\varphi\mu_n \rightarrow\int_\Omega \varphi d\mu,~~\forall \varphi \in C_0(\bar{\Omega}).$$  
\end{definition}
\noindent We now construct the following sequence of problems which is the approximating problem to ($P_\beta$). 
\begin{equation}\label{4p2}\tag{$P_{\beta,n}$}
\begin{split}
(-\Delta)^su_n-\alpha \frac{u_n}{|x|^{2s}}&=\lambda u_n+ \frac{1}{(u_n+\frac{1}{n})^\gamma}+ \beta \left(\int_{\Omega}\frac{u_n^{2_b^*}(y)}{|x-y|^b}dy\right)u_n^{2_b^*-1}+\mu_n ~\text{in}~\Omega,\\
u_n&>0~\text{in}~\Omega,\\
u_n&= 0~\text{in}~\mathbb{R}^N\setminus\Omega,
\end{split}
\end{equation}
\begin{definition}\label{weak form of apprpx}
Let $\mu_n\in L^\infty(\Omega)$ and $\gamma\in(0,1)$. Then  $u_n\in H_0^s(\Omega)$ is said to be a weak solution to ($P_{\beta,n}$) if
	\begin{align}\label{approximating weak}
	\int_{\mathbb{R}^{N}}(-\Delta)^{s/2}u_n&\cdot(-\Delta)^{s/2}\phi-\alpha\int_\Omega\frac{u_n\phi}{|x|^{2s}}\nonumber\\&=\lambda\int_{\Omega}u_n\phi+\int_{\Omega}\frac{\phi}{(u_n+1/n)^\gamma}+\beta\int_{\Omega}\int_{\Omega} \frac{u_n^{2_b^*}(y)u_n^{2_b^*-1}(x)\phi(x)}{|x-y|^b}dxdy+\int_{\Omega}\mu_n \phi	
	\end{align}
	for every $\phi\in C_c^\infty(\Omega)$. Further, for any $\omega\subset \subset\Omega$ there exists $C_\omega$ such that $u_n\geq C_\omega>0$.
\end{definition} 
\noindent Denote 
$$H=\{u\in H_0^s(\Omega):\|u\|_{L^{2_s^*}(\Omega)}=1\}.$$
We now state the following existence theorem. We prove this in the next section.
\begin{theorem}\label{approx solution}
Let $\gamma\in (0,1)$, $b<\min\{4s,N\}$, $\lambda\in (0,\lambda_1)$ and $\alpha\in (0, m_\lambda C_H)$ such that $\left(\frac{N-b+2s}{2N-b}\right)^{\frac{N-b+2s}{2N-b}}\leq m_\lambda-\frac{\alpha}{C_H}$. Then there exists $\bar{\beta}\in (0,\infty)$ such that for every $\beta\in (0,\bar{\beta})$, $(P_{\beta,n})$ possesses a positive weak solution $u_n$ in $H$.
\end{theorem} 
\noindent We are now ready to state the main result of the paper.
\begin{theorem}\label{main theorem}
	Let the assumptions on $\gamma, b, \alpha,\lambda$ are same as in Theorem $\ref{approx solution}$. Then there exists $0<\bar{\beta}<\infty$ such that for $\beta\in(0,\bar{\beta})$ the problem $(P_\beta)$ has a positive SOLA $u\in W_{0}^{\bar{s},m}(\Omega)$, in the sense of  Definition $\ref{SOLA}$, for every $\bar{s}<s$ and $m<\frac{N}{N-s}$.	
\end{theorem}
\section{Existence of weak solution to $(P_{\beta,n})$- Proof of Theorem $\ref{approx solution}$}\label{sec 2}
\noindent We prove the existence of weak solution to $(P_{\beta,n})$ via two sequence of problems given below, $(P^1_{\beta,n})$ and $(P^2_{\beta,n})$. Let us first consider the following problem
\begin{equation}\label{4p3}\tag{$P^1_{\beta,n}$}
\begin{split}
(-\Delta)^sw_n-\alpha \frac{w_n}{|x|^{2s}}&=\lambda w_n+ \frac{1}{(w_n+\frac{1}{n})^\gamma}+\mu_n ~\text{in}~\Omega,\\
w_n&>0~\text{in}~\Omega,\\
w_n&= 0~\text{in}~\mathbb{R}^N\setminus\Omega.
\end{split}
\end{equation}
We now prove the problem $(P^1_{\beta,n})$ has a weak solution in $\bar{H}=\{u\in H_0^s(\Omega):\|u\|_{L^{2_s^*}(\Omega)}<1\}$ in the following lemma. 
\begin{lemma}\label{weak first}
	Let $\gamma\in (0,1)$, $\lambda\in (0,\lambda_1)$ and $\alpha\in (0, m_\lambda C_H)$. Then the problem $(P_{\beta,n}^1)$ admits a positive weak solution $w_n$ in $\bar{H}$.
	\begin{proof}
		Consider the Euler -Lagrange functional $J_{n}$ associated to problem $(P_{\beta,n}^1)$, i.e.
		\begin{align*}
		J_n(w_n)&=\frac{1}{2}\left(\|w_n\|^2_{H_0^s(\Omega)}-\alpha\|w_n\|_{NH}^2-\lambda\|w_n\|^2_{L^2(\Omega)}\right)-\frac{1}{1-\gamma}\int_{\Omega}\left((w_n+1/n)^{1-\gamma}-\frac{1}{n^{1-\gamma}}\right)-\int_{\Omega}\mu_n w_n\nonumber\\&=\mathbb{H}_{\alpha,\lambda}(w_n)-\frac{1}{1-\gamma}\int_{\Omega}\left((w_n+1/n)^{1-\gamma}-\frac{1}{n^{1-\gamma}}\right)-\int_{\Omega}\mu_n w_n
		\end{align*}
		for any $w_n\in \bar{H}$. By Corollary $\ref{coercive}$, $\mathbb{H}_{\alpha,\lambda}$ is coercive and weakly lower semi continuous in $H_0^s(\Omega)$. Since $\gamma\in (0,1)$, it is easy to prove that $J_n$ is coercive and weakly lower semi continuous in $\bar{H}$. Thus, by using a standard minimization argument we show the existence of a minimum $w_n$ for $J_n$ on $\bar{H}$ and hence a weak solution to  $(P_{\beta,n}^1)$ in $\bar{H}$.\\
		Fix $\bar{w}_n\in \bar{H}\setminus\{0\}$. Then for sufficiently small $t>0$,
		\begin{align*}
		J_n(t\bar{w}_n)&=t^2\mathbb{H}_{\alpha,\lambda}(\bar{w}_n)-\frac{1}{1-\gamma}\int_{\Omega}\left((t\bar{w}_n+1/n)^{1-\gamma}-\frac{1}{n^{1-\gamma}}\right)-t\int_{\Omega}\mu_n \bar{w}_n\\&<0.
		\end{align*} 
		This implies 
		$$J_n(w_n)=\underset{\bar{w}_n\in \bar{H}}{\min}J_n(\bar{w}_n)<0=J_n(0)$$
		and hence $w_n$ is nontrivial. Let us consider the following problem 
		\begin{equation}\label{helping}
		\begin{split}
		(-\Delta)^s \underline{w}_n&= \frac{1}{(\underline{w}_n+\frac{1}{n})^\gamma}+\mu_n ~\text{in}~\Omega,\\
		\underline{w}_n&>0~\text{in}~\Omega,\\
		\underline{w}_n&= 0~\text{in}~\mathbb{R}^N\setminus\Omega.
		\end{split}
		\end{equation}
		According to Lemma 2.3 and Lemma 2.4 of Ghosh et.al \cite{Ghosh}, $\eqref{helping}$ admits a nontrivial positive weak solution in $H_0^s(\Omega)$ and for any $\omega\subset \subset\Omega$, there exists $C_\omega$ such that $\underline{w}_n\geq C_\omega>0$. Using a standard comparison principle, Lemma 2.4 of \cite{Abdellaoui}, we conclude that $\underline{w}_n\leq w_n$. Thus, $w_n\geq \underline{w}_n\geq C_\omega>0$. This finishes the proof.
	\end{proof}
\end{lemma}
\noindent Let us consider the second problem.
\begin{equation}\label{4p4}\tag{$P^2_{\beta,n}$}
\begin{split}
(-\Delta)^sv_n-\alpha\frac{v_n}{|x|^{2s}}+f_n(x,v_n)&=\lambda v_n+\beta\left(\int_{\Omega}\frac{(w_n+v_n)^{2_b^*}(y)}{|x-y|^b}dy\right)(w_n+v_n)^{2_b^*-1}~\text{in}~\Omega,\\
v_n&>0~\text{in}~\Omega,\\
v_n&= 0~\text{in}~\mathbb{R}^N\setminus\Omega,
\end{split}
\end{equation}
where $w_n$ is the positive weak solution of $(P^1_{\beta,n})$ obtained from Lemma $\ref{weak first}$ and the function $f_n:\Omega\times\mathbb{R}\rightarrow \mathbb{R}\cup\{-\infty\}$ is given by
\begin{eqnarray}\label{gn}
\begin{split}
f_n(x,s)=&\begin{cases}
\frac{1}{(w_n(x)+\frac{1}{n})^\gamma}-\frac{1}{(s+w_n(x)+\frac{1}{n})^\gamma}&\text{if}~ s+w_n(x)+\frac{1}{n}>0\\ 
-\infty&\text{otherwise}.
\end{cases}
\end{split}
\end{eqnarray}
For $(x,s)\in \Omega\times\mathbb{R}$, let us denote $F_n(x,s)=\int_{0}^{s}f_n(x,\tau)d\tau$. The corresponding energy functional $J_{\beta,n}:H_0^s(\Omega)\rightarrow (-\infty,\infty]$ of $(P^2_{\beta,n})$ is defined by
\begin{eqnarray}
\begin{split}
J_{\beta,n}(v_n)=&\begin{cases}
\frac{1}{2}\left(\int_{\mathbb{R}^{2N}}\frac{|v_n(x)-v_n(y)|^2}{|x-y|^{N+2s}}dxdy-\alpha\int_{\Omega}\frac{v_n^2}{|x|^{2s}}dx-\lambda\int_{\Omega}v_n^2dx\right)\\+\int_{\Omega}F_n(x,v_n)dx-\frac{\beta}{22_s^*}\int_{\Omega}\int_{\Omega}\frac{(v_n+w_n)^{2_b^*}(x)(v_n+w_n)^{2_b^*}(y)}{|x-y|^b}dxdy&\text{if}~ F_n(.,v_n)\in L^1(\Omega)\\ 
\infty&\text{otherwise}.
\end{cases}
\end{split}
\end{eqnarray}
Further, 
\begin{align}
\langle J^\prime_{\beta,n}(v_n),v\rangle&=\int_{\mathbb{R}^{2N}}\frac{(v_n(x)-v_n(y))(v(x)-v(y)}{|x-y|^{N+2s}}dxdy-\alpha\int_{\Omega}\frac{v_n v}{|x|^{2s}}dx-\lambda\int_{\Omega}v_n vdx\nonumber\\&~~~~+\int_{\Omega}f_n(x,v_n)vdx-\beta\int_{\Omega}\int_{\Omega}\frac{(v_n+w_n)^{2_b^*}(y)(v_n+w_n)^{2_b^*-1}(x)v(x)}{|x-y|^b}dxdy\nonumber
\end{align}
for any $v\in H_0^s(\Omega)$. Let us denote $$H_n=\{u\in \bar{H}:\|u+w_n\|_{L^{2_s^*}(\Omega)}=1\}.$$
\begin{definition}
We say a function $v_n\in H_n$ is a weak solution of $(P^2_{\beta,n})$ if $v_n$ is a critical point of the energy functional $J_{\beta,n}$.	
\end{definition}
\begin{lemma}\label{P-S}
	The functional $J_{\beta,n}$ satisfies the Palais-Smale (P-S) condition in $H_n$ for energy level
	$$c<\frac{1}{2}\left(\frac{N-b+2s}{2N-s}\right)\frac{\mathbb{S}_{C,b}^{\frac{2N-b}{N-b+2s}}}{\beta^{\frac{N-2s}{N-b+2s}}}-\frac{C(N,b)\beta}{22_b^*},$$
	where $\mathbb{S}_{C,b}$ is the best Sobolev constant defined in $\eqref{best sobolev choquard}$ and $C(N,b)$ is the sharp constant in the Hardy-Littlewood-Sobolev Inequality given in $\eqref{sharp constant}$.
	\begin{proof}
		Consider a (P-S) sequence $(v_{n,j})$ of $J_{\beta,n}$ in $H_n$, i.e. $J_{\beta,n}(v_{n,j})\rightarrow c$ and $J^\prime_{\beta,n}(v_{n,j})\rightarrow 0$ as $j\rightarrow\infty$. It is easy to show that  $J_{\beta,n}$ is coercive when restricted to $H_n$. Thus, $(v_{n,j})$ is bounded in $H_0^s(\Omega)$ and there exists $v_n$ in $H_0^s(\Omega)$ such that, up to a sub sequential level, $v_{n,j}\rightarrow v_n$ weakly in $H_0^s(\Omega)$. We now claim the following.\\
		{\it Claim:} $v_{n,j}\rightarrow v_n$ strongly in $H_0^s(\Omega)$ and $v_n\in H_n$.\\
		Let $\|v_{n,j}-v_n\|_{H_0^s(\Omega)}^2\rightarrow a^2$ and $\int_{\Omega}\int_{\Omega}\frac{(v_{n,j}-v_n)^{2_b^*}(x)(v_{n,j}-v_n)^{2_b^*}(y)}{|x-y|^b}dxdy\rightarrow d^{22_b^*}$ as $j\rightarrow\infty$. Thus, we have
		\begin{align}
		&\langle J^\prime_{\beta,n}(v_{n,j})-J^\prime_{\beta,n}(v_n),v_{n,j}-v_n\rangle\nonumber\\&=\|v_{n,j}-v_n\|_{H_0^s(\Omega)}^2-\alpha\|v_{n,j}-v_n\|_{NH}^2-\lambda\|v_{n,j}-v_n\|_{L^2(\Omega)}^2+\int_\Omega(f_n(x,v_{n,j})-f_n(x,v_n))(v_{n,j}-v_n)\nonumber\\&~~~~-\beta\int_{\Omega}\int_{\Omega}\frac{(v_{n,j}+w_n)^{2_b^*}(v_{n,j}+w_n)^{2_b^*-1}(v_{n,j}-v_n)}{|x-y|^b}dxdy\nonumber\\&~~~~+\beta\int_{\Omega}\int_{\Omega}\frac{(v_{n}+w_n)^{2_b^*}(v_{n}+w_n)^{2_b^*-1}(v_{n,j}-v_n)}{|x-y|^b}dxdy.\nonumber
		\end{align}
		This implies,
		\begin{align}
&\langle J^\prime_{\beta,n}(v_{n,j})-J^\prime_{\beta,n}(v_n),v_{n,j}-v_n\rangle\nonumber\\&\geq \left(m_\lambda-\frac{\alpha}{C_H}\right) \|v_{n,j}-v_n\|_{H_0^s(\Omega)}^2-\beta\int_{\Omega}\int_{\Omega}\frac{(v_{n,j}-v_n)^{2_b^*}(v_{n,j}-v_n)^{2_b^*}+(v_{n}+w_n)^{2_b^*}(v_{n}+w_n)^{2_b^*}}{|x-y|^b}dxdy\nonumber\\&~~~~+\beta\int_{\Omega}\int_{\Omega}\frac{(v_{n,j}+w_n)^{2_b^*}(v_{n,j}+w_n)^{2_b^*-1}(v_{n}+w_n)}{|x-y|^b}dxdy\nonumber\\&~~~~+\int_\Omega(f_n(x,v_{n,j})-f_n(x,v_n))(v_{n,j}-v_n)+\beta\int_{\Omega}\int_{\Omega}\frac{(v_{n}+w_n)^{2_b^*}(v_{n}+w_n)^{2_b^*-1}(v_{n,j}-v_n)}{|x-y|^b}dxdy.\nonumber
		\end{align}
		On using the Brezis-Lieb Lemma \cite{Brezis}, Hardy-Littlewood-Sobolev Inequality (Proposition $\ref{HLS}$), Corollary $\ref{coercive}$ and passing the limit $j\rightarrow \infty$ in the above equation we get
		\begin{equation}\label{first ineq}
		\beta d^{22_b^*}\geq \left(m_\lambda-\frac{\alpha}{C_H}\right)a^2.
		\end{equation}
		From $\eqref{best sobolev choquard}$ we already have $a^2\geq \mathbb{S}_{C,b}d^2$. Thus, by simplification we obtain
		\begin{equation}\label{second ineq}
		d\geq \left(\frac{(m_\lambda-\frac{\alpha}{C_H})\mathbb{S}_{C,b}}{\beta}\right)^{\frac{N-2s}{2(N-b+2s)}}.
		\end{equation}
	We have	the sequence $(v_{n,j})$ is a (P-S) sequence in $H_n$ and by the choice of $\alpha, \lambda$, clearly $\left(\frac{N-b+2s}{2N-b}\right)^{\frac{N-b+2s}{2N-b}}\leq m_\lambda-\frac{\alpha}{C_H}$. Now applying Corollary $\ref{coercive}$, $\eqref{sharp constant}$, $\eqref{first ineq}$ and $\eqref{second ineq}$ we have,
		\begin{align}
		c&=\lim\limits_{j\rightarrow\infty}J_{\beta,n}(v_{n,j})\nonumber\\&=	\frac{1}{2}\lim\limits_{j\rightarrow\infty}\left(\int_{\mathbb{R}^{2N}}\frac{|v_{n,j}(x)-v_{n,j}(y)|^2}{|x-y|^{N+2s}}dxdy-\alpha\int_{\Omega}\frac{v_{n,j}^2}{|x|^{2s}}dx-\lambda\int_{\Omega}v_{n,j}^2dx\right)\nonumber\\&~~~~+\lim
		\limits_{j\rightarrow\infty}\left(\int_{\Omega}F_n(x,v_{n,j})dx-\frac{\beta}{22_s^*}\int_{\Omega}\int_{\Omega}\frac{(v_{n,j}+w_n)^{2_b^*}(x)(v_{n,j}+w_n)^{2_b^*}(y)}{|x-y|^b}dxdy\right)\nonumber\\&\geq  \lim\limits_{j\rightarrow\infty}\left(\frac{1}{2}\left(m_\lambda-\frac{\alpha}{C_H}\right)\|v_{n,j}\|_{H_0^s(\Omega)}^2-\frac{C(N,b)\beta}{22_b^*}\|v_{n,j}+w_n\|^{22_b^*}_{L^{2^*_s}(\Omega)}\right)\nonumber\\&\geq\lim\limits_{j\rightarrow\infty}\frac{1}{2}\left(m_\lambda-\frac{\alpha}{C_H}\right)\mathbb{S}_{C,b}\left(\int_{\Omega}\int_{\Omega}\frac{v_{n,j}^{2_b^*}(x)v_{n,j}^{2_b^*}(y)}{|x-y|^b}dxdy\right)^{\frac{2}{22^*_b}}-\frac{C(N,b)\beta}{22_b^*}\nonumber\\&\geq \frac{1}{2}\left(m_\lambda-\frac{\alpha}{C_H}\right)d^2\mathbb{S}_{C,b}-\frac{C(N,b)\beta}{22_b^*}\nonumber\\&\geq\frac{1}{2}\left(m_\lambda-\frac{\alpha}{H}\right)\mathbb{S}_{C,b}\left(\frac{(m_\lambda-\frac{\alpha}{C_H})\mathbb{S}_{C,b}}{\beta}\right)^{\frac{N-2s}{N-b+2s}}-\frac{C(N,b)\beta}{22_b^*}\nonumber\\&=\frac{1}{2}\frac{(\mathbb{S}_{C,b}(m_\lambda-\frac{\alpha}{C_H}))^{\frac{2N-b}{N-b+2s}}}{\beta^{\frac{N-2s}{N-b+2s}}}-\frac{C(N,b)\beta}{22_b^*}\nonumber\\&\geq \frac{1}{2}\left(\frac{N-b+2s}{2N-b}\right)\frac{\mathbb{S}_{C,b}^{\frac{2N-b}{N-b+2s}}}{\beta^{\frac{N-2s}{N-b+2s}}}-\frac{C(N,b)\beta}{22_b^*}.\nonumber
		\end{align}
		This is a contradiction to our assumption $$c<\frac{1}{2}\left(\frac{N-b+2s}{2N-b}\right)\frac{\mathbb{S}_{C,b}^{\frac{2N-b}{N-b+2s}}}{\beta^{\frac{N-2s}{N-b+2s}}}-\frac{C(N,b)\beta}{22_b^*}.$$ 
		Thus, $a=0$ and $\lim\limits_{j\rightarrow\infty}\|v_{n,j}-v_n\|_{H_0^s(\Omega)}=0$. Hence, the claim.  
	\end{proof}	
\end{lemma}
\noindent Let us consider the following sequence $(Z_\epsilon)$ given by 
$$Z_\epsilon=\epsilon^{-\frac{N-2s}{2}}\mathbb{S}_{s,2}^{\frac{(N-b)(2s-N)}{4(N-b+2s)}} C(N,b)^{\frac{2s-N}{2(N-b+2s)}} z^*\left(\frac{x}{\epsilon}\right), ~~x\in \mathbb{R}^N,$$
where $z^*(x)=\bar{z}\left(\frac{x}{\mathbb{S}_{2,s}^{\frac{1}{2s}}}\right)$, $\bar{z}(x)=\frac{\tilde{z}(x)}{\|\tilde{z}\|_{L^{2_s^*}(\Omega)}}$ and $\tilde{z}(x)=\eta_1(\eta_2^2+|x|^2)^{-\frac{N-2s}{2}}$,  $\eta_1\in \mathbb{R}^N\setminus\{0\},~\eta_2>0$. By Lemma $\ref{best constant achieve}$, for each $\epsilon>0$, corresponding $Z_\epsilon$ satisfies the problem
$$(-\Delta)^s z=(|x|^{-b}* |z|^{2_b^*})|z|^{2_b^*-2}v~~\text{ in }\mathbb{R}^N.$$
Let us assume $0\in\Omega$. Consider $\xi\in C_c^{\infty}(\mathbb{R}^N)$ such that $0\leq\xi\leq 1$, for fixed $\delta>0$, $B_{4\delta}\subset\Omega$, $\xi\equiv 0$ in $\mathbb{R}^N\setminus B_{2\delta}$, $\xi\equiv1$ in $B_\delta$. Define  $$\Phi_\epsilon(x)=\xi(x)Z_\epsilon(x).$$ Then $\Phi_\epsilon=0$ in $\mathbb{R}^N\setminus\Omega$. By Proposition 6.2 of Giacomoni et al. \cite{Giacomoni arxiv}, there exists $a_1,a_2,a_3,a_4>0$ such that for $1<q<\min\{2,\frac{N}{N-2s}\}$ we have the following four estimates.
\begin{align}
\int_{\mathbb{R}^{2N}}\frac{|\Phi_\epsilon(x)-\Phi_\epsilon(y)|^2}{|x-y|^{N+2s}}dxdy&\leq \mathbb{S}_{C,b}^{\frac{2N-b}{N-b+2s}}+a_1\epsilon^{N-2s},\nonumber\\
\int_{\Omega}\int_{\Omega}\frac{|\Phi_\epsilon|^{2_b^*}(x)|\Phi_\epsilon|^{2_b^*}(y)}{|x-y|^b}dxdy&\geq \mathbb{S}_{C,b}^{\frac{2N-b}{N-b+2s}}-a_2\epsilon^N,\nonumber\\
\int_{\Omega}|\Phi_\epsilon|^{q}dx&\leq a_3\epsilon^{(N-2s)q/2},\nonumber\\\int_{\Omega}\int_{\Omega}\frac{|\Phi_\epsilon|^{2_b^*}(x)|\Phi_\epsilon|^{2_b^*}(y)}{|x-y|^b}dxdy&\leq \mathbb{S}_{C,b}^{\frac{2N-b}{N-b+2s}}+a_4\epsilon^N.\nonumber
\end{align}
\begin{lemma}\label{upper bound}
	There exists $\bar{\beta}>0$ such that for $\beta\in (0,\bar{\beta})$ and for $\epsilon>0$ sufficienty small,
	$$\sup\{J_{\beta,n}(t\Phi_\epsilon): t\geq 0\}<\frac{1}{2}\left(\frac{N-b+2s}{2N-b}\right)\frac{\mathbb{S}_{C,b}^{\frac{2N-b}{N-b+2s}}}{\beta^{\frac{N-2s}{N-b+2s}}}-\frac{C(N,b)\beta}{22_b^*}.$$
	\begin{proof}
		Clearly for $\beta<\left(\frac{2_b^*}{C(N,b)}\left(\frac{N-b+2s}{2N-b}\right)\right)^{\frac{N-b+2s}{2N-b}}\mathbb{S}_{C,b}$, we have $$\frac{1}{2}\left(\frac{N-b+2s}{2N-b}\right)\frac{\mathbb{S}_{C,b}^{\frac{2N-b}{N-b+2s}}}{\beta^{\frac{N-2s}{N-b+2s}}}-\frac{C(N,b)\beta}{22_b^*}>0.$$
	For a fixed sufficiently small $\epsilon>0$ and for any $t\geq0$, 
		\begin{align}
		J_{\beta,n}(t\Phi_\epsilon)&=\frac{t^2}{2}\left(\int_{\mathbb{R}^{2N}}\frac{|\Phi_\epsilon(x)-\Phi_\epsilon(y)|^2}{|x-y|^{N+2s}}dxdy-\alpha\int_{\Omega}\frac{|\Phi_\epsilon|^2}{|x|^{2s}}-\lambda\int_{\Omega}|\Phi_\epsilon|^2\right)+\int_{\Omega}F_n(x,t\Phi_\epsilon)dx\nonumber\\&~~~~-\frac{\beta}{22_b^*}\int_{\Omega}\int_{\Omega}\frac{|t\Phi_\epsilon+w_n|^{2_b^*}|t\Phi_\epsilon+w_n|^{2_b^*}}{|x-y|^b}dxdy\nonumber\\&\leq\frac{t^2}{2}\int_{\mathbb{R}^{2N}}\frac{|\Phi_\epsilon(x)-\Phi_\epsilon(y)|^2}{|x-y|^{N+2s}}dxdy+\int_{\Omega}\frac{|t\Phi_\epsilon|}{(w_n+1/n)^\gamma}\nonumber\\&~~~-\frac{1}{1-\gamma}\int_{\Omega}(t\Phi_\epsilon+w_n+1/n)^{1-\gamma}-(w_n+1/n)^{1-\gamma}\nonumber\\&~~~~-\frac{\beta}{22_b^*}\int_{\Omega}\int_{\Omega}\frac{|t\Phi_\epsilon+w_n|^{2_b^*}|t\Phi_\epsilon+w_n|^{2_b^*}}{|x-y|^b}dxdy\nonumber\\&\leq \frac{t^2}{2}(\mathbb{S}_{C,b}^{\frac{2N-b}{N-b+2s}}+a_1\epsilon^{N-2s})+tn^\gamma\int_{\Omega}|\Phi_\epsilon|+\frac{C(N,b)\beta}{22_b^*}-\frac{C(N,b)\beta}{22_b^*}\nonumber\\&~~~-\frac{1}{1-\gamma}\int_{\Omega}(t\Phi_\epsilon+w_n+1/n)^{1-\gamma}-(w_n+1/n)^{1-\gamma}-\frac{\beta t^{22_b^*}}{22_b^*}\int_{\Omega}\int_{\Omega}\frac{|\Phi_\epsilon|^{2_b^*}|\Phi_\epsilon|^{2_b^*}}{|x-y|^b}dxdy\nonumber\\&\leq \frac{t^2}{2}(\mathbb{S}_{C,b}^{\frac{2N-b}{N-b+2s}}+a_1\epsilon^{N-2s})+tn^\gamma a_3^{1/q}\epsilon^{(N-2s)/2}+\frac{C(N,b)\beta}{22_b^*}-\frac{C(N,b)\beta}{22_b^*}\nonumber\\&~~~-\frac{1}{1-\gamma}\int_{\Omega}(t\Phi_\epsilon+w_n+1/n)^{1-\gamma}-(w_n+1/n)^{1-\gamma}-\frac{\beta t^{22_b^*}}{22_b^*}(\mathbb{S}_{C,b}^{\frac{2N-b}{N-b+2s}}-a_2\epsilon^N).
		\end{align}
		Assume $\beta\leq1$ and define  $g:\mathbb{R}^+\rightarrow\mathbb{R}$ as
		\begin{align}
		g(t)&=\frac{C(N,b)\beta}{22_b^*}-\frac{1}{1-\gamma}\int_{\Omega}(t\Phi_\epsilon+w_n+1/n)^{1-\gamma}-(w_n+1/n)^{1-\gamma}\nonumber\\&\leq \frac{C(N,b)}{22_b^*}-\frac{1}{1-\gamma}\int_{\Omega}(t\Phi_\epsilon+w_n+1/n)^{1-\gamma}-(w_n+1/n)^{1-\gamma}\nonumber\\&\leq \frac{C(N,b)}{22_b^*}-\frac{1}{1-\gamma}\int_{\Omega}(t\Phi_\epsilon)^{1-\gamma} +C.
		\end{align} 
		Hence, $g(t)\rightarrow -\infty$ as $t\rightarrow\infty$. Therefore, there exists $\bar{t}>0$ such that $g(t)\leq0$ for every $t\geq \bar{t}$. First consider the case $t\geq \bar{t}$ and we have
		\begin{align}
		J_{\beta,n}(t\Phi_\epsilon)&\leq \frac{t^2}{2}(\mathbb{S}_{C,b}^{\frac{2N-b}{N-b+2s}}+a_1\epsilon^{N-2s})+tn^\gamma a_3^{1/q}\epsilon^{(N-2s)/2}-\frac{\beta t^{22_b^*}}{22_b^*}(\mathbb{S}_{C,b}^{\frac{2N-b}{N-b+2s}}-a_2\epsilon^N)-\frac{C(N,b)\beta}{22_b^*}\nonumber\\&= \bar{g}_\epsilon(t).\nonumber
		\end{align}
		Clearly, $\bar{g}_\epsilon$ attains the maximum value at $$t_\beta=\left(\frac{1}{\beta}\right)^{\frac{2(N-b+2s)}{N-2s}}+o(\epsilon^{(N-2s)/2}).$$
		This implies,
		\begin{align}
		J_{\beta,n}(t\Phi_\epsilon)&\leq \frac{1}{2}\left(\frac{N-b+2s}{2N-b}\right)\frac{\mathbb{S}_{C,b}^{\frac{2N-b}{N-b+2s}}}{\beta^{\frac{N-2s}{N-b+2s}}}-\frac{C(N,b)\beta}{22_b^*}+o(\epsilon^{(N-2s)/2})\nonumber\\&< \frac{1}{2}\left(\frac{N-b+2s}{2N-b}\right)\frac{\mathbb{S}_{C,b}^{\frac{2N-b}{N-b+2s}}}{\beta^{\frac{N-2s}{N-b+2s}}}-\frac{C(N,b)\beta}{22_b^*}.
		\end{align}
		For the second case, i.e. for $t<\bar{t}$,
		\begin{align}
		J_{\beta,n}(t\Phi_\epsilon)&\leq\frac{t^2}{2}\int_{\mathbb{R}^{2N}}\frac{|\Phi_\epsilon(x)-\Phi_\epsilon(y)|^2}{|x-y|^{N+2s}}dxdy+\int_{\Omega}\frac{|t\Phi_\epsilon|}{(w_n+1/n)^\gamma}\nonumber\\&\leq \frac{t^2}{2}(\mathbb{S}_{C,b}^{\frac{2N-b}{N-b+2s}}+a_1\epsilon^{N-2s})+tn^\gamma a_3^{1/q}\epsilon^{(N-2s)/2}\nonumber\\&<\frac{\bar{t}^2}{2}(\mathbb{S}_{C,b}^{\frac{2N-b}{N-b+2s}}+a_1\epsilon^{N-2s})+\bar{t}n^\gamma a_3^{1/q}\epsilon^{(N-2s)/2}.\nonumber
		\end{align}
		Choose $\beta^*>0$ depending on $\bar{t},N,s,\mathbb{S}_{C,b}$ such that for $\beta\in(0,\beta^*)$ we get 
		$$J_{\beta,n}(t\Phi_\epsilon)<\frac{1}{2}\left(\frac{N-b+2s}{2N-b}\right)\frac{\mathbb{S}_{C,b}^{\frac{2N-b}{N-b+2s}}}{\beta^{\frac{N-2s}{N-b+2s}}}-\frac{C(N,b)\beta}{22_b^*}.$$
	Denote $\bar{\beta}=\min\{1,\left(\frac{2_b^*}{C(N,b)}\left(\frac{N-b+2s}{2N-b}\right)\right)^{\frac{N-b+2s}{2N-b}}\mathbb{S}_{C,b},\beta^*\}$. Thus, for $\beta\in(0,\bar{\beta})$ we obtain $$\sup\{J_{\beta,n}(t\Phi_\epsilon): t\geq 0\}<\frac{1}{2}\left(\frac{N-b+2s}{2N-b}\right)\frac{\mathbb{S}_{C,b}^{\frac{2N-b}{N-b+2s}}}{\beta^{\frac{N-2s}{N-b+2s}}}-\frac{C(N,b)\beta}{22_b^*}.$$
		This concludes the proof.
	\end{proof}
\end{lemma}
\noindent The following is the existence theorem for $(P^2_{\beta,n})$ in $H_n$.
\begin{theorem}\label{second solution}
Let the assumptions on $\gamma,b,\lambda,\alpha$ are same as in Theorem $\ref{approx solution}$. Then there exists $\bar{\beta}>0$ such that for every $\beta\in(0,\bar{\beta})$, $(P^2_{\beta,n})$ admits a positive weak solution $v_n\in H_n$. 
\begin{proof}
The functional $J_{\beta,n}$ is bounded from below and  G$\hat{a}$teaux-differentiable on $H_n$.  Hence, it satisfies all the hypotheses of Theorem $\ref{eke}$, i.e. Ekeland variational principle. Thus, we can produce a Palais-Smale sequence $(v_{n,j})$ in $H_n$ of the functional $J_{\beta,n}$. By  Lemma $\ref{P-S}$ and Lemma $\ref{upper bound}$, $(v_{n,j})$ satisfies the (P-S) conditions and hence, up to a sub sequential level, $(v_{n,j})$ converges strongly to $v_n\in H_n$. This implies $v_n$ is a critical point of $J_{\beta,n}$ and therefore a weak solution of $(P_{\beta,n})$ in $H_n$ for any $\beta\in (0,\bar{\beta})$. 
\end{proof}
\begin{proof}[Proof of Theorem $\ref{approx solution}$]
According to Theorem $\ref{second solution}$, $v_n$ is a nontrivial weak solution to $(P^2_{\beta,n})$ in $H_n=\{u\in H_0^s(\Omega):\|u+w_n\|_{L^{2_s^*}(\Omega)}=1\}$, for every for $\beta\in(0,\bar{\beta})$, where $w_n$ is the weak solution of $(P^1_{\beta,n})$ from Lemma $\ref{weak first}$. Hence, the function $u_n=v_n+w_n$ is a positive weak solution of $(P_{\beta,n})$ in $H=\{u\in H_0^s(\Omega):\|u\|_{L^{2_s^*}(\Omega)}=1\}$ in the sense of Definition $\ref{weak form of apprpx}$.
\end{proof}
\end{theorem}
\section{Existence of SOLA to $(P_\beta)$- Proof of Theorem $\ref{main theorem}$}\label{sec 3}
From the previous section, Section $\ref{sec 2}$, we have a weak solution $u_n$ of the approximating problem $(P_{\beta,n})$ in $H$. In this section, using some apriori estimates, we pass the limit $n\rightarrow \infty$ in the weak formulation of $(P_{\beta,n})$, i.e. in $\eqref{approximating weak}$, to obtain a SOLA to $(P_\beta)$. 
\begin{lemma}\label{lemma1}
	Let $u_n$ be a weak solution to $(P_{\beta,n})$ in $H$. Then the sequence $(u_n)$ is uniformly bounded in $W_0^{\bar{s},m}(\Omega)$ for every $\bar{s}<s$ and $m<\frac{N}{N-s}$.
	\begin{proof}
		 Let us fix a $k>0$ and define a truncation function $T_k:\mathbb{R}\rightarrow \mathbb{R}$ by
		\begin{eqnarray}
		\begin{split}
		T_k(u_n)=&\begin{cases}
		u_n&\text{if}~ u_n\leq k\\ 
		k &\text{if}~ u_n>k.\nonumber 
		\end{cases}
		\end{split}
		\end{eqnarray}
Choose $\phi=T_k(u_n)$ in $\eqref{approximating weak}$ as a test function. Thus, we have
		\begin{align}\label{bdd}
		\int_{\mathbb{R}^N}|(-\Delta)^{s/2}T_k(u_n)|^2&\leq\int_{\mathbb{R}^{N}}(-\Delta)^{s/2}u_n\cdot(-\Delta)^{s/2}T_k(u_n)\nonumber\\&=\alpha\int_\Omega\frac{u_n T_k(u_n)}{|x|^{2s}}+\lambda\int_\Omega u_n T_k(u_n)+\beta\int_\Omega\int_{\Omega}\frac{ u_n^{2_b^*}u_n^{2_b^*-1}T_k(u_n)}{|x-y|^b}dxdy
		\nonumber\\&~~~~	+\int_{\Omega}\frac{1}{(u_n+\frac{1}{n})^\gamma}T_k(u_n)+\int_{\Omega}\mu_n T_k(u_n).
		\end{align}
	The sequence $(u_n)\subset H$, i.e. $\|u_n\|_{L^{2_s^*}(\Omega)}=1$ for each $n$ and  $(\mu_n)$ is bounded in $L^1(\Omega)$. So, using Proposition $\ref{HLS}$, the above equation $\eqref{bdd}$ becomes
		\begin{align}\label{4p6}
		\int_{\mathbb{R}^N}|(-\Delta)^{s/2}T_k(u_n)|^2&\leq \alpha k\int_\Omega\frac{u_n}{|x|^{2s}}+\lambda\int_\Omega u_n^2+\int_{\Omega} u_n^{1-\gamma}+k\|\mu_n\|_{L^1(\Omega)}+\beta C(N,b)\nonumber\\&\leq \alpha k\int_\Omega\frac{u_n}{|x|^{2s}}+\lambda C_1\|u_n\|^2_{L^{2_s^*}(\Omega)}+ C_2\|u_n\|_{L^{2_s^*}(\Omega)}^{1-\gamma}+C_3k+\beta C(N,b)\nonumber\\&\leq \alpha k\int_\Omega\frac{u_n}{|x|^{2s}}+ C_4k.
		\end{align}
		Since $N>2s$, this implies $(N-2s(2_s^*)^\prime)>0$ where $(2_s^*)^\prime$ is the H\"{o}lder conjugate of $2_s^*$. Therefore, using the H\"{o}lder's inequality we get
		\begin{align}\label{hardy bounded}
		\int_\Omega\frac{u_n}{|x|^{2s}}&\leq\left\|x|^{-2s}\right\|_{L^{(2_s^*)^\prime}(\Omega)}\|u_n\|_{L^{2_s^*}(\Omega)}\nonumber\\&=\left\||x|^{-2s}\right\|_{L^{(2_s^*)^\prime}(\Omega)}\nonumber\\&\leq C_5.
		\end{align}
		Hence, $(u_n)$ is bounded in $L^1(\Omega,|x|^{-2s})$ and from $\eqref{4p6}$, we conclude that
		\begin{equation}\label{4pp6}
		\int_{\mathbb{R}^N}|(-\Delta)^{s/2}T_k(u_n)|^2\leq Ck.
		\end{equation}
		Therefore, $(T_k(u_n))$ is uniformly bounded in $H_0^s(\Omega)$. By following the proof of Lemma 4.1 of Panda et al. \cite{Panda 1}, we conclude that $(u_n)$ is uniformly bounded in $W_0^{\bar{s},m}(\Omega)$, for all $\bar{s}<s$ and $m<\frac{N}{N-s}$.
	\end{proof}
\end{lemma}
\noindent We are now in a position to prove our main result, i.e. the existence of positive SOLA to $(P_\beta)$.
\begin{proof}[Proof of Theorem $\ref{main theorem}$]
	Let $\mu\in \mathcal{M}(\Omega)$ and the assumptions on $\gamma,b,\alpha,\lambda$ are same as provided in the statement of Theorem $\ref{main theorem}$. From Theorem $\ref{approx solution}$, for any $\beta\in (0,\bar{\beta})$, we have a positive weak solution $u_n$ of $(P_{\beta,n})$ in $H$. According to Lemma $\ref{lemma1}$, $(u_n)$ is a uniformly bounded sequence in $W_0^{\bar{s},m}(\Omega)$ for every $\bar{s}<s$ and $m<\frac{N}{N-s}$. Thus, there exist a sub sequence of $(u_n)$, still denoted as $(u_n)$, and $u\in W_0^{\bar{s},m}(\Omega)$ such that $u_n\rightarrow u$ weakly in $W_0^{\bar{s},m}(\Omega)$. This implies $u_n\rightarrow u$ a.e. in $\mathbb{R}^N$ and $u\equiv 0$ in $\mathbb{R}^N\setminus\Omega$. From the weak formulation of $(P_{\beta,n})$, i.e. from $\eqref{approximating weak}$, we have
	\begin{align}\label{approximating weak new}
	\int_{\mathbb{R}^{2N}}&\frac{(u_n(x)-u_n(y))(\phi(x)-\phi(y))}{|x-y|^{N+2s}}dxdy -\alpha\int_\Omega\frac{u_n\phi}{|x|^{2s}}\nonumber\\&=\lambda\int_\Omega u_n \phi+\int_{\Omega}\frac{1}{(u_n+\frac{1}{n})^\gamma}\phi+\beta\int_\Omega\int_{\Omega}\frac{ u_n^{2_b^*}u_n^{2_b^*-1}\phi}{|x-y|^b}dxdy+\int_{\Omega}\mu_n \phi,
	\end{align}
	for all $\phi\in C_c^\infty(\Omega)$. Proceeeding on similar lines as in Theorem 1.1. of \cite{Panda 1} and using Vitali convergence theorem, we establish
	$$\lim\limits_{n\rightarrow\infty}\int_{\mathbb{R}^{2N}}\frac{(u_n(x)-u_n(y))(\phi(x)-\phi(y))}{|x-y|^{N+2s}}dxdy=\int_{\mathbb{R}^{2N}}\frac{(u(x)-u(y))(\phi(x)-\phi(y))}{|x-y|^{N+2s}}dxdy.$$
On using the definition of convergence in measure, i.e. Definition $\ref{measure}$, Dominated convergence theorem and the fact that $\|u_n\|_{L^{2_s^*}(\Omega)}=1$, we can pass the limit $n\rightarrow \infty$ in the following integrals.
 $$\lim\limits_{n\rightarrow\infty}\int_{\Omega}\mu_n\phi=\int_{\Omega}\phi d\mu,$$
	$$\lim\limits_{n\rightarrow\infty}\beta\int_\Omega\int_{\Omega}\frac{ u_n^{2_b^*}u_n^{2_b^*-1}\phi}{|x-y|^b}dxdy=\beta\int_\Omega\int_{\Omega}\frac{ u^{2_b^*}u^{2_b^*-1}\phi}{|x-y|^b}dxdy,$$
	$$\lim\limits_{n\rightarrow\infty}\int_\Omega\frac{u_n\phi}{|x|^{2s}}=\int_\Omega\frac{u\phi}{|x|^{2s}},$$
	$$\lim\limits_{n\rightarrow\infty}\int_{\Omega}\frac{1}{(u_n+\frac{1}{n})^\gamma}\phi=\int_{\Omega}\frac{1}{u^\gamma}\phi.$$
Hence, we obtain a SOLA $u$ to $(P_\beta)$, in the sense of Definition $\ref{SOLA weak}$, as the limit of approximation in  $\eqref{approximating weak new}$. Thus, for every $\phi\in C_c^\infty(\Omega)$, $u$ satisfies
	\begin{align}
	\int_{\mathbb{R}^{2N}}&\frac{(u(x)-u(y))(\phi(x)-\phi(y))}{|x-y|^{N+2s}}dxdy-\alpha\int_\Omega\frac{u\phi}{|x|^{2s}}\nonumber\\&=\lambda\int_\Omega u \phi+\int_{\Omega}\frac{1}{u^\gamma}\phi+\beta\int_\Omega\int_{\Omega}\frac{ u^{2_b^*}u^{2_b^*-1}\phi}{|x-y|^b}dxdy+\int_{\Omega}\mu \phi.\nonumber
	\end{align}
This concludes the proof of our main result.
\end{proof}
\section*{Acknowledgement}
The author Akasmika Panda thanks the financial assistantship received from the Ministry of Human
Resource Development (M.H.R.D.), Govt. of India. Both the authors also acknowledge the facilities received from the Department of mathematics, National Institute of Technology Rourkela.
 

\begin{thebibliography}{99}
	\bibitem{Abdellaoui}B. Abdellaoui and R. Bentifour, Caffarelli–Kohn–Nirenberg type inequalities of fractional order and applications, Journal of Functional Analysis, 272, 3998-4029, 2017.
	\bibitem{Adimurthy} Adimurthi and J. Giacomoni, Multiplicity of positive solutions for a singular and critical elliptic problem in $\mathbb{R}^2$, Commun. Contemp. Math. 8 (5), 621-656, 2006.
	\bibitem{Applebaum}D. Applebaum, L\'{e}vy processes-from probability to finance and quantum groups, Notices Amer. Math. Soc., 51(11), 1336–1347, 2004.
	\bibitem{Barrios}B. Barrios, M. Medina, and I. Peral, Some remarks on the solvability of nonlocal
		elliptic problems with the Hardy potential, Commun. Contemp. Math., 16,
		1350046, 29 pp, 2014.
	\bibitem{Boccardo 2}L. Boccardo, T. Gallouet and L. Orsina, Existence and uniqueness of entropy solutions
	for nonlinear elliptic equations involving measure data, Ann. Inst. H. Poincar\'{e} Anal. Non Lineaire 13, 539-551, 1996.
\bibitem{Boccardo}	L. Boccardo and L. Orsina, Semilinear elliptic equations with singular nonlinearities,
	Calc. Var., 37, 363-380, 2010.	
	\bibitem{Boccardo 1}L. Boccardo and T. Gallouët, Nonlinear elliptic and parabolic equations involving measure data. J. Funct. Anal. 87, 149-169, 1989.
	\bibitem{Brezis}H. Brezis and E. Lieb, A relation between pointwise convergence of functions and convergence of functionals, Proc. Amer. Math. Soc., 88 (1983), 486-90.
	\bibitem{Buffoni}B. Buffoni, L. Jeanjean and C. A. Stuart, Existence of a nontrivial solution to a strongly
	indefinite semilinear equation, Proc. Amer. Math. Soc., 119(1),179-186, 1993.
	\bibitem{Canino}A. Canino, L. Montoro, B. Sciunzi and M. Squassina, Nonlocal problems with singular nonlinearity, Bull. Sci. math. 2017.
	\bibitem{Chen}Y. H. Chen, C. Liu, Ground state solutions for non-autonomous fractional Choquard equations, Nonlinearity 29, 1827–1842, 2016.
	\bibitem{Ekeland}I. Ekeland, On the variational principle, J. Math. Anal. Appl., 47, 1974, 324-353.
	\bibitem{Fall}M. M. Fall, Semilinear elliptic equations for the fractional Laplacian with Hardy potential,
	 arXiv:1109.5530v4 [math.AP] 24 Oct 2012..
	\bibitem{Fiscella}A. Fiscella and P. Pucci, On Certain Nonlocal Hardy-Sobolev Critical Elliptic Dirichlet Problems, Advances in Differential Equations, 21 (5-6), 571-599, 2016.
	\bibitem{Gao}F. Gao and M. Yang, On the Brezis-Nirenberg type critical problem for nonlinear Choquard equation, Sci. China Math., 61, 1219-1242, 2018.
	\bibitem{Ghanmi} A. Ghanmi and K. Saoudi, The Nehari manifold for a singular elliptic equation involving the fractional Laplace operator, Fractional Differential Calculus, 6(2), 201-217, 2016.
	\bibitem{Ghosh}S. Ghosh, D. Choudhuri and R. K. Giri, Singular Nonlocal Problem Involving Measure Data, Bull. Braz. Math. Soc., 50(1), 187-209, 2018.
	\bibitem{Giacomoni}J. Giacomoni, T. Mukherjee and K. Sreenadh, Positive solutions of fractional elliptic equation with critical and singular nonlinearity, Adv. Nonlinear Anal., 6 (3), 327-354, 2017.
	\bibitem{Giacomoni Hardy}J. Giacomoni, T. Mukherjee, K. Sreenadh, Doubly nonlocal system with Hardy–Littlewood–Sobolev critical nonlinearity, J. Math.Anal.Appl., 467, 638–672, 2018.
	\bibitem{Giacomoni arxiv}J. Giacomoni,
	D. Goel and K. Sreenadh, Singular doubly nonlocal elliptic problems with Choquard type
	critical growth nonlinearities, arXiv:2002.02937v1 [math.AP] 7 Feb 2020.
	\bibitem{Haitao}Y. Haitao, Multiplicity and asymptotic behavior of positive solutions for a singular semilinear
	elliptic problem, J. Differential Equations, 189, 487-512, 2003.
	\bibitem{Kuusi} T. Kuusi, G. Mingione and Y. Sire, Nonlocal Equations with Measure Data, Commun. Math. Phys. 337, 1317-1368, 2015.
	\bibitem{Lazer} A. C. Lazer and P. J. McKenna, On a singular nonlinear elliptic boundary-value problem, Proc. Amer. Math. Soc., 111(3), 721-730, 1991.
	\bibitem{Choquard}E. Lieb and M. Loss,``Analysis”, Graduate Studies in Mathematics, AMS, Providence,
	Rhode island, 2001.
	\bibitem{Liu}D. L\"{u}, G. Xu, On nonlinear fractional Schr\"{o}dinger equations with Hartree-type nonlinearity, Appl. Anal. 97(2), 255–273, 2018.
	\bibitem{Valdinocci}E. Di Nezza, G. Palatucci, E. Valdinoci, Hitchhiker’s guide to the fractional Sobolev spaces. Bull. Sci. Math. 136(5), 521-573 (2012).
	\bibitem{Panda} A. Panda, S. Ghosh and D. Choudhuri, Elliptic Partial Differential Equation Involving a Singularity and a Radon Measure, The Journal of the Indian Mathematical Society, 86 (1-2), 95-117, 2019. \bibitem{Panda 1} A. Panda, D. Choudhuri, R. K. Giri, Fractional elliptic problem involving a singularity, a critical exponent and a Radon measure, arXiv:2002.11393v1 [math.AP] 26 Feb 2020.
	\bibitem{Pekar}S. Pekar, Untersuchung uber die Elektronentheorie der Kristalle, Akademie Verlag,
	Berlin, 1954.
	\bibitem{Sevadei 1}R. Servadei and E. Valdinoci, The Brezis–Nirenberg result for the fractional Laplacian, Trans. Amer. Math. Soc. 367, 67-102, 2015.
	\bibitem{Shieh}T. T. Shieh and D. Spector. On a new class of fractional partial differential equations. Adv. Calc. Var., 8(4), 321-336, 2015.
	\bibitem{Soudi}K. Saoudi, S. Ghosh and D. Choudhuri, Multiplicity and H\"{o}lder regularity of solutions for a nonlocal elliptic PDE involving singularity, J. Math. Phys., 60, 101509, 2019.
	\bibitem{Sun}Y. Sun and D. Zhang, The role of the power 3 for elliptic equations with negative exponents, Calculus of Variations, 49, 909-922., 2014.
	\bibitem{Tang}X. Tang and S. Chen, Singularly perturbed Choquard equations with nonlinearity satisfying Berestycki-Lions assumptions, Adv. Nonlinear Anal., 9, no. 1, 413-437, 2020.
	\bibitem{Wang} Y. Wang and Y. Yang, Bifurcation results for the critical
	Choquard problem involving fractional
	p-Laplacian operator, Boundary Value Problems, 2018:132, 2018. 
\end{thebibliography}
\end{document}